\documentclass[10pt]{article}

\usepackage{amsmath}
\usepackage{amsthm}
\usepackage{amssymb}
\usepackage{latexsym}
\usepackage{pstricks}
\usepackage{graphicx, pst-plot, pst-node, pst-text, pst-tree}
\usepackage{titlefoot}
\usepackage{enumerate}
\usepackage{titlesec}
\usepackage{cancel}
\usepackage{rotating}
\usepackage[small,it]{caption}
\usepackage{color}
\definecolor{lightgray}{rgb}{0.8, 0.8, 0.8}
\definecolor{darkgray}{rgb}{0.7, 0.7, 0.7}
\definecolor{darkblue}{rgb}{0, 0, .4}

\usepackage[bookmarks]{hyperref}
\hypersetup{
        colorlinks=true,
        linkcolor=darkblue,
        anchorcolor=darkblue,
        citecolor=darkblue,
        urlcolor=darkblue,
        pdfpagemode=UseThumbs,
        pdftitle={A stack and a pop stack in series},
        pdfsubject={Combinatorics},
        pdfauthor={Rebecca Smith and Vincent Vatter},
}

\newtheorem{theorem}{Theorem}[section]
\newtheorem{proposition}[theorem]{Proposition}

\newtheorem{conjecture}[theorem]{Conjecture}
\newtheorem{question}[theorem]{Question}

\newtheoremstyle{example}{\topsep}{\topsep}%
     {}%         Body font
     {}%         Indent amount (empty = no indent, \parindent = para indent)
     {\bfseries}% Thm head font
     {.}%        Punctuation after thm head
     {.5em}%     Space after thm head (\newline = linebreak)
     {\thmname{#1}\thmnumber{ #2}}%         Thm head spec
\theoremstyle{example}

\newtheoremstyle{negexample}{\topsep}{\topsep}%
     {}%         Body font
     {}%         Indent amount (empty = no indent, \parindent = para indent)
     {\bfseries}% Thm head font
     {.}%        Punctuation after thm head
     {.5em}%     Space after thm head (\newline = linebreak)
     {\thmname{#1}\thmnumber{ #2}}%         Thm head spec
\theoremstyle{negexample}

\newcounter{todocounter}

\long\def\symbolfootnote[#1]#2{\begingroup%
\def\thefootnote{\fnsymbol{footnote}}\footnote[#1]{#2}\endgroup}

\makeatletter
\newcommand{\Rm}[1]{\expandafter\@slowromancap\romannumeral #1@}
\newfont{\footsc}{cmcsc10 at 8truept}
\newfont{\footbf}{cmbx10 at 8truept}
\newfont{\footrm}{cmr10 at 10truept}
\renewenvironment{abstract}%
                {
                  \begin{list}{}%
                     {\setlength{\rightmargin}{1in}%
                      \setlength{\leftmargin}{1in}}%
                   \item[]\ignorespaces\begin{small}}%
                 {\end{small}\unskip\end{list}}

\amssubj{05A17, 06A07}

\newpagestyle{main}[\small]{
        \headrule
        \sethead[\usepage][][]
        {\sc A Stack and a Pop Stack in Series}{}{\usepage}}

\title{\sc{A Stack and a Pop Stack in Series}}
\author{Rebecca Smith\footnote{The first author was partially supported by the NSA Young Investigator Grant H98230-08-1-0100.}
\\[-0.25ex]
\small Department of Mathematics\\[-0.5ex]
\small SUNY Brockport\\[-0.5ex]
\small Brockport, New York\\[15pt]
Vincent Vatter\footnote{The second author was partially supported by the NSA Young Investigator Grant H98230-12-1-0207.}\\[-0.25ex]
\small Department of Mathematics\\[-0.5ex]
\small University of Florida\\[-0.5ex]
\small Gainesville, Florida\\[-1.5ex]
}

\titleformat{\section}
        {\large\sc}
        {\thesection.}{1em}{}   

\date{}

\begin{document}
\maketitle

\pagestyle{main}

\newcommand{\s}{\mathbf{s}}
\newcommand{\m}{\mathbf{m}}
\renewcommand{\t}{\mathbf{t}}
\renewcommand{\b}{\mathbf{b}}
\newcommand{\f}{\mathbf{f}}
\newcommand{\rev}{\operatorname{rev}}
\newcommand{\dual}{\operatorname{dual}}
\newcommand{\C}{\mathcal{C}}
\newcommand{\D}{\mathcal{D}}
\newcommand{\Av}{\operatorname{Av}}

% Stack operations:
\newcommand{\inp}{\textsf{i}}
\newcommand{\tra}{\textsf{t}}
\newcommand{\out}{\textsf{o}}

% Cohabitation graphs:
\newcommand{\R}{\stackrel{R}{\sim}}
\renewcommand{\L}{\stackrel{L}{\sim}}
\newcommand{\notR}{\stackrel{R}{\not\sim}}
\newcommand{\notL}{\stackrel{L}{\not\sim}}

% The OEIS links:
\newcommand{\OEISlink}[1]{\href{http://oeis.org/#1}{#1}}
\newcommand{\OEISref}{\href{http://oeis.org/}{OEIS}~\cite{sloane:the-on-line-enc:}}
\newcommand{\OEIS}[1]{(Sequence \OEISlink{#1} in the \OEISref.)}

%
%
%
%
% Drawing stacks
%
%
%
%

\def\sdwys #1{\xHyphenate#1$\wholeString}
\def\xHyphenate#1#2\wholeString {\if#1$%
\else\say{\ensuremath{#1}}\hspace{2pt}%
\takeTheRest#2\ofTheString
\fi}
\def\takeTheRest#1\ofTheString\fi
{\fi \xHyphenate#1\wholeString}
\def\say#1{\begin{turn}{-90}\ensuremath{#1}\end{turn}}

\newenvironment{twostacks}
{% Begin code
	\begin{scriptsize}
	\psset{xunit=0.0355in, yunit=0.0355in, linewidth=0.02in}
	\begin{pspicture}(0,-2)(35,20)
	\psline{c-c}(0,15)(10,15)(10,0)(15,0)(15,15)(20,15)(20,0)(25,0)(25,15)(35,15)
	\rput[l](-0.5,12.5){\mbox{output}}
	\rput[r](35,12.5){\mbox{input}}
}
{% End code
	\end{pspicture}
	\end{scriptsize}
}

\newcommand{\fillstack}[4]{%
	\rput[l](-0.5,17.5){\ensuremath{#1}}
	\rput[c](12.6, 7.5){\begin{sideways}{\sdwys{#2}}\end{sideways}}
	\rput[c](22.6, 7.5){\begin{sideways}{\sdwys{#3}}\end{sideways}}
	\rput[r](35,17.5){\ensuremath{#4}}
}

\newcommand{\stackinput}{%
	\psline[linecolor=darkgray]{c->}(27, 17.5)(22.5, 17.5)(22.5, 14)
}
\newcommand{\stackshortinput}{%
	\psline[linecolor=darkgray]{c->}(26, 17.5)(22.5, 17.5)(22.5, 14)
}
\newcommand{\stacktransfer}{%
	\psline[linecolor=darkgray]{c->}(22.5, 14)(22.5, 17.5)(12.5, 17.5)(12.5, 14)
}
\newcommand{\stackoutput}{%
	\psline[linecolor=darkgray]{c->}(12.5, 14)(12.5, 17.5)(8, 17.5)
}
\newcommand{\stackshortoutput}{%
	\psline[linecolor=darkgray]{c->}(12.5, 14)(12.5, 17.5)(9, 17.5)
}
\newcommand{\stacksupershortoutput}{%
	\psline[linecolor=darkgray]{c->}(12.5, 14)(12.5, 17.5)(9.9, 17.5)
}

\begin{abstract}
We study sorting machines consisting of a stack and a pop stack in series, with or without a queue between them.  While there are, a priori, four such machines, only two are essentially different: a pop stack followed directly by a stack, and a pop stack followed by a queue and then by a stack.  In the former case, we obtain complete answers for the basis and enumeration of the sortable permutations.  In the latter case, we present several conjectures.
\end{abstract}

\section{Introduction}

A stack is a last-in first-out sorting device with push and pop operations.  In Volume 1 of \emph{The Art of Computer Programming}~\cite[Section 2.2.1]{knuth:the-art-of-comp:1}, Knuth showed that the permutation $\pi$ can be sorted (meaning that by applying push and pop operations to the sequence of entries $\pi(1),\dots,\pi(n)$ one can output the sequence $1,\dots,n$) if and only if $\pi$ avoids the permutation $231$, i.e., if and only if there do not exist three indices $1\le i_1<i_2<i_3\le n$ such that $\pi(i_1),\pi(i_2),\pi(i_3)$ are in the same relative order as $231$.  Shortly thereafter Tarjan~\cite{tarjan:sorting-using-n:}, Even and Itai~\cite{even:queues-stacks-a}, Pratt~\cite{pratt:computing-permu:}, and Knuth himself in Volume 3~\cite[Section 5.2.4]{knuth:the-art-of-comp:3} studied networks with multiple stacks in series or in parallel.  The questions typically studied for such networks include:
\begin{itemize}
\item Can the set of sortable permutations be characterized by a finite set of forbidden patterns (e.g., $\{231\}$ in the case of a single stack)?
\item How many permutations of each length can be sorted?
\end{itemize}
For $k\ge 2$ stacks in series or in parallel the answer to the first question is no, due to Murphy~\cite{murphy:restricted-perm:} and Tarjan~\cite{tarjan:sorting-using-n:}, respectively.  The exact enumeration question appears to be much less tractable, and here only relatively crude bounds are known; see Albert, Atkinson, and Linton~\cite{albert:permutations-ge:}.

Given how difficult the two stacks in series problem appears to be, numerous researchers have considered weaker machines.  Atkinson, Murphy, and Ru\v{s}kuc~\cite{atkinson:sorting-with-tw:} considered sorting with two {\it increasing\/} stacks in series, i.e., two stacks whose entries must be in increasing order when read from top to bottom\footnote{Even without this restriction, the final stack must be increasing if the sorting is to be successful.}.  They characterized the permutations this machine can sort with an infinite list of forbidden patterns, and also found the enumeration of these permutations.  (Interesting, these permutations are in bijection with the $1342$-avoiding permutations previously counted by B\'ona~\cite{bona:exact-enumerati:}.)  Another weakening, sorting with a stack of depth $2$ followed by a standard stack (of infinite depth), was studied by Elder~\cite{elder:permutations-ge:}.  He characterized the sortable permutations with a finite list of forbidden patterns, but did not enumerate these permutations.

The objects we study, pop stacks, were introduced by Avis and Newborn~\cite{avis:on-pop-stacks-i:}.  A {\it pop stack\/} is a handicapped form of a stack in which the only way to output an entry in the stack is to pop the entire stack (in last-in first-out order as usual).  Avis and Newborn considered placing pop stacks in series, which by their interpretation means that when the entire set of items currently in the $i$th pop stack is popped, they are pushed immediately onto the $(i+1)$st pop stack.  They proved that the set of permutations sortable by $k$ pop stacks in series can be characterized by a finite set of forbidden patterns and provided the enumeration of these permutations for every $k$.
 
There is another way to view sorting with pop stacks in series, where one is allowed to save the output of one pop stack and pass it into the next an entry at a time.  Serially linking pop stacks in this manner corresponds to placing a queue between the pop stacks, and is much more powerful than the Avis-Newborn interpretation.  Atkinson and Stitt~\cite{atkinson:restricted-perm:wreath}, who also gave a simpler derivation of Avis and Newborn's enumerative results using what is now known as the substitution decomposition, found the (rational) generating function for the permutations that can be sorted by a pop stack followed by a queue followed by another pop stack.

Here we consider sorting with a stack and a pop stack in series.  A priori, there are three different methods that these may be connected:

\begin{tabular}{ll}
\\
\textsf{PS}:&A pop stack followed by a stack, connected in Avis and Newborn's manner.\\
\textsf{PQS}:&A pop stack followed by a queue followed by a stack.\\
\textsf{SP}:&A stack followed by a pop stack, connected in Avis and Newborn's manner.\\
\textsf{SQP}:&A stack followed by a queue followed by a pop stack.\\
\\
\end{tabular}

Clearly the permutations sortable by \textsf{PS} are a subset of those sortable by \textsf{PQS}.  In the next section, we prove that \textsf{SP} and \textsf{SQP} are equivalent, and that \textsf{PQS} is a symmetry of these two machines.  Then, in Section~\ref{sec-pop-without} we characterize and enumerate the permutations sortable by \textsf{PS}.  In Section~\ref{sec-pop-with} we consider the class \textsf{PQS}, but are able to establish few concrete results.

It will be helpful to give names to the operations involved in these sorting machines.  Given a system of two stacks (of any type) in series, we refer to moving an entry from the input to the first stack as an \emph{input}, moving an entry from the first stack to the second stack as a \emph{transfer}, and moving an entry from the second stack to the output as an \emph{output}.  When there is a queue between the two stacks, we use the term transfer to describe both moving an entry from the first stack to the queue and moving it from the queue to the second stack.

We conclude the introduction with a bit of terminology which will be useful.  A \emph{permutation class} is a downset of permutations under the containment order.  Every permutation class can be specified by the set of minimal permutations which are \emph{not} in the class, which we call its \emph{basis}.  Finally, for a set $B$ of permutations, we denote by $\Av(B)$ the class of permutations which do not contain any element of $B$.  For example, Knuth's result says that the stack-sortable permutations are precisely $\Av(231)$, i.e., that they have the singleton basis $\{231\}$.  Given any naturally defined sorting machine, the set of sortable permutations forms a class\footnote{The exception that proves this rule is West's notion of $2$-stack-sortability~\cite{west:sorting-twice-t:}, which has an unusual defect due to restrictions on how the machine can use its two stacks.  For example, this machine can sort $35241$, but not its subpermutation $3241$.}.

Finally, it is frequently helpful to remember that the permutation containment order has eight symmetries which form the dihedral group of the square.  These are generated by two symmetries \emph{inverse} and \emph{reverse}, defined, respectively, by
\begin{eqnarray*}
\pi^{-1}(\pi(i))&=&i,\\
\pi^{\textrm{r}}(i)&=&\pi(n+1-i),\\
\end{eqnarray*}
for all $i$.  We will also make use of the \emph{complement} symmetry, defined by
\begin{eqnarray*}
\pi^{\textrm{c}}(i)&=&n+1-\pi(i),
\end{eqnarray*}
also for all $i$.

\section{The Equivalence of \textsf{PQS}, \textsf{SP}, \& \textsf{SQP}}
% The \textsf{PQS}/\textsf{SP} argument comes from page 313 of Max's thesis.

We begin with the easier equivalence, between the machines \textsf{SP} and \textsf{SQP}.

\begin{proposition}\label{prop-equiv-SP-SQP}
The machines \textsf{SP} and \textsf{SQP} are equivalent.
\end{proposition}
\begin{proof}
More generally, adding a queue after a regular stack never alters the sorting capabilities of a permutation machine.  The machine \textsf{SM} (a stack followed by \textsf{M}) can sort the entries of $\pi$ if and only if it can sort $\pi(i_1),\dots,\pi(i_n)$ where $i_1\cdots i_n$ is a permutation which can be \emph{generated} by a stack starting with the identity permutation as input.  Clearly, the same is true for the machine \textsf{SQM}, because the queue cannot alter the order in which entries pass through it.  (While it is not important to this proof, note that a stack can reorder the entries of $\pi$ as $\pi(i_1),\dots,\pi(i_n)$ if and only if $i_1\cdots i_n$ avoids $231^{-1}=312$.)
\end{proof}

Now we show that the permutations sortable by \textsf{PQS} and \textsf{SQP} are symmetries of each other.  To do so we consider a construction introduced by Murphy in his thesis~\cite{murphy:restricted-perm:}.  The \emph{two-stack dual} of the permutation $\pi$ of length $n$ is defined by
$$
\pi^{\rm d}=\left(\left(\pi^{\rm r}\right)^{-1}\right)^{\rm r},
$$
or more concretely by
$$
\pi^{\rm d}(i)=n+1-\pi^{-1}(n+1-i)
$$
for all $i$.

\begin{proposition}[Murphy~\cite{murphy:restricted-perm:}]\label{prop-two-stack-dual}
The permutation $\pi$ can be sorted by two stacks in series if and only if the permutation $\pi^{\rm d}$ can be sorted by two stacks in series.
\end{proposition}

\begin{proof}
Consider any sequence of pushes, transfers, and pops which sorts the permutation $\pi$.  We think of these operations as generating $12\cdots n$ from the input $\pi$.  By performing these operations in reverse, we obtain a procedure to generate $n\cdots 21$ from $\pi^{\rm r}$ --- in this new procedure, the last entry popped from the second stack becomes the first entry pushed onto the first stack.  By applying $(\pi^{\rm r})^{-1}$ to the input symbols, we obtain a procedure to generate the identity from $(\pi^{\rm r})^{-1}\circ n\cdots 21$, i.e., sort the two-stack dual of $\pi$.
\end{proof}

%An example might be nice.

So long as there is a queue in between, the proof of Proposition~\ref{prop-two-stack-dual} shows that the order of the stack and pop stack can be interchanged if we change the permutation to its two-stack dual, giving the result below.

\begin{proposition}
The permutation $\pi$ can be sorted by \textsf{PQS} if and only if $\pi^{\rm d}$ can be sorted by $\textsf{SQP}=\textsf{SP}$.
\end{proposition}

\section{Without a Queue Between --- \textsf{PS}}\label{sec-pop-without}

\begin{figure}[t]
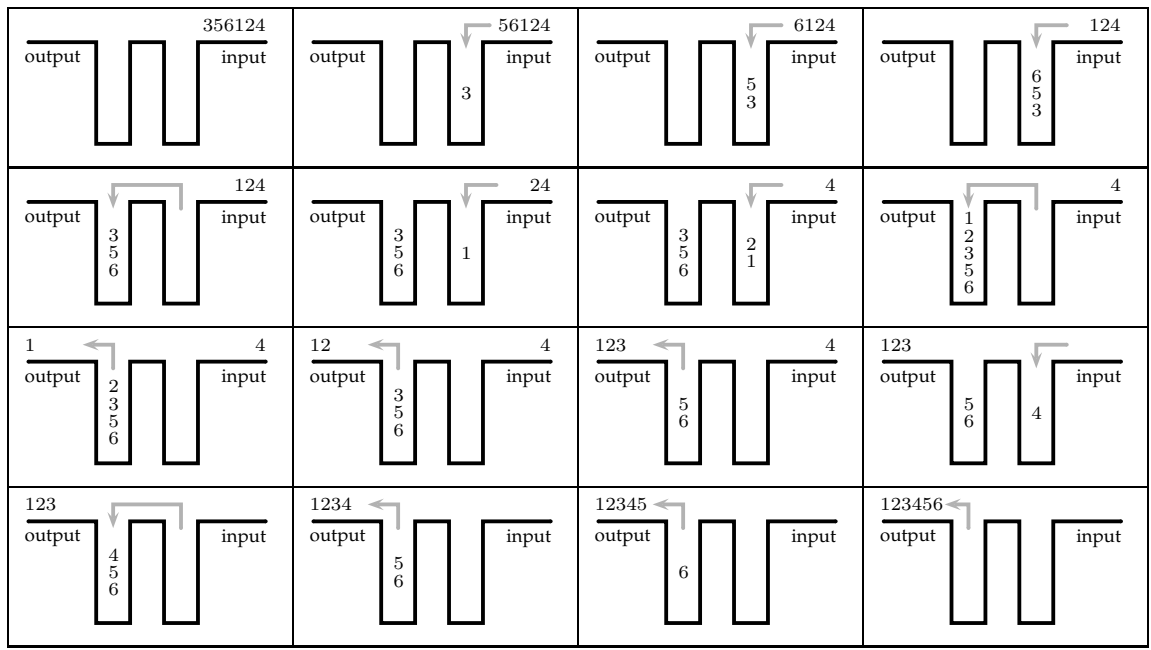

\begin{center}

\begin{tabular}{|c|c|c|c|}
\hline
\begin{twostacks}
\fillstack{}{}{}{356124}
\end{twostacks}
&
\begin{twostacks}
\fillstack{}{}{3}{56124}
\stackshortinput
\end{twostacks}
&
\begin{twostacks}
\fillstack{}{}{35}{6124}
\stackinput
\end{twostacks}
&
\begin{twostacks}
\fillstack{}{}{356}{124}
\stackinput
\end{twostacks}
\\\hline
\begin{twostacks}
\fillstack{}{653}{}{124}
\stacktransfer
\end{twostacks}
&
\begin{twostacks}
\fillstack{}{653}{1}{24}
\stackinput
\end{twostacks}
&
\begin{twostacks}
\fillstack{}{653}{12}{4}
\stackinput
\end{twostacks}
&
\begin{twostacks}
\fillstack{}{65321}{}{4}
\stacktransfer
\end{twostacks}
\\\hline
\begin{twostacks}
\fillstack{1}{6532}{}{4}
\stackoutput
\end{twostacks}
&
\begin{twostacks}
\fillstack{12}{653}{}{4}
\stackoutput
\end{twostacks}
&
\begin{twostacks}
\fillstack{123}{65}{}{4}
\stackoutput
\end{twostacks}
&
\begin{twostacks}
\fillstack{123}{65}{4}{}
\stackinput
\end{twostacks}
\\\hline
\begin{twostacks}
\fillstack{123}{654}{}{}
\stacktransfer
\end{twostacks}
&
\begin{twostacks}
\fillstack{1234}{65}{}{}
\stackoutput
\end{twostacks}
&
\begin{twostacks}
\fillstack{12345}{6}{}{}
\stackoutput
\end{twostacks}
&
\begin{twostacks}
\fillstack{123456}{}{}{}
\stackshortoutput
\end{twostacks}
\\\hline
\end{tabular}

\caption{Sorting the permutation $24513$ with the \textsf{PS} machine.}
\label{fig-PS}
\end{center}
\end{figure}

We begin by finding three permutations that the \textsf{PS} machine \emph{cannot} sort.  We will later show that these are the only minimal permutations which cannot be sorted by \textsf{PS}.

\begin{proposition}\label{prop-PS-no-3}
The permutations $2431$, $3142$, and $3241$ are not \textsf{PS}-sortable.
\end{proposition}
\begin{proof}
The proof consists of three separate case analyses.  As the cases are similar, we give the details for the first only.  The other two follow from the work of Smith~\cite{smith:a-decreasing-st:} mentioned in the conclusion or can be viewed as exercises for the reader.

Consider attempting to sort the permutation $2431$ with \textsf{PS}.  First the $2$ must be pushed into the pop stack.  Suppose the $2$ is not transferred to the next stack before the $4$ enters the pop stack.  If the entries of the pop stack ever contain a increase when read from top to bottom, then the sorting will clearly fail, so the $4$ and thus also (by the pop property) the $2$ must be transferred to the stack at this point.  However, now there is no way to output the $1$  before the $3$ is forced to be transferred to the stack above the $2$.

Alternatively, if the $2$ is transferred from the pop stack to the stack before the $4$ enters the pop stack, again there is no way output the $1$ before the $3$ and $4$ are transferred to the stack above the $2$.
\end{proof}

Thus the \textsf{PS}-sortable permutations are a subclass of $\Av(2431,3142)$.  The structure of the reverse-complement of this class, $\Av(3142,4213)$, was described by Albert, Atkinson, and Vatter~\cite{albert:inflations-of-g:2x4}.  This structural description rests on the notion of \emph{simple permutations}; the only intervals that are mapped to intervals by such permutations are singletons and their entire domains.  For example, $31542$ is not simple because it maps $\{3,4\}$ to $\{4,5\}$, but $25314$ is simple.  Simple permutations are precisely those that do not arise from a non-trivial inflation, in the following sense.  Given a permutation $\sigma$ of length $m$ and nonempty permutations $\alpha_1,\dots,\alpha_m$, the \emph{inflation} of $\sigma$ by $\alpha_1,\dots,\alpha_m$,  denoted $\sigma[\alpha_1,\dots,\alpha_m]$, is the permutation of length $|\alpha_1|+\cdots+|\alpha_m|$ obtained by replacing each entry $\sigma(i)$ by an interval that is order isomorphic to $\alpha_i$ in such a way that the intervals are order isomorphic to $\sigma$.  For example,
\[
2413[1,132,321,12]=4\ 798\ 321\ 56. 
\]
In particular, the inflation $12[\alpha_1,\alpha_2]$ is called \emph{(direct) sum} and denoted by $\alpha_1\oplus\alpha_2$.  A permutation is \emph{sum decomposable} if it can be expressed as a nontrivial sum, and \emph{sum indecomposable} otherwise.  A sum decomposable permutation can always be expressed as $\alpha_1\oplus\alpha_2$ where $\alpha_1$ is sum \emph{in}decomposable.  The inflation $21[\alpha_1,\alpha_2]$ is similarly called \emph{skew sum} and denoted $\alpha_1\ominus\alpha_2$; we define the terms \emph{skew decomposable} and \emph{indecomposable} analogously.

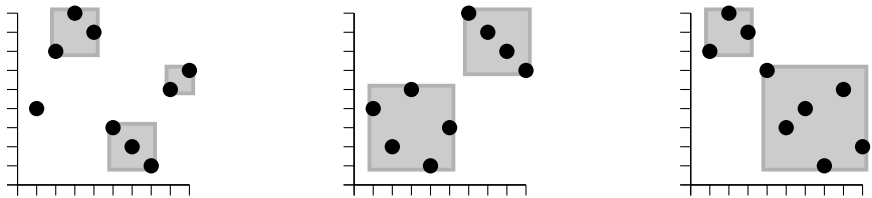
\begin{figure}
\begin{center}
\begin{tabular}{ccccc}
\psset{xunit=0.01in, yunit=0.01in}
\psset{linewidth=0.005in}
\begin{pspicture}(0,0)(93,93)
\psaxes[dy=10, Dy=1, dx=10, Dx=1, tickstyle=bottom, showorigin=false, labels=none](0,0)(90,90)
\psframe[linecolor=darkgray,fillstyle=solid,fillcolor=lightgray,linewidth=0.02in](7,37)(13,43)
\psframe[linecolor=darkgray,fillstyle=solid,fillcolor=lightgray,linewidth=0.02in](17,67)(43,93)
\psframe[linecolor=darkgray,fillstyle=solid,fillcolor=lightgray,linewidth=0.02in](47,7)(73,33)
\psframe[linecolor=darkgray,fillstyle=solid,fillcolor=lightgray,linewidth=0.02in](77,47)(93,63)
\pscircle*(10,40){0.04in}
\pscircle*(20,70){0.04in}
\pscircle*(30,90){0.04in}
\pscircle*(40,80){0.04in}
\pscircle*(50,30){0.04in}
\pscircle*(60,20){0.04in}
\pscircle*(70,10){0.04in}
\pscircle*(80,50){0.04in}
\pscircle*(90,60){0.04in}
\end{pspicture}
&\rule{0.5in}{0pt}&
\psset{xunit=0.01in, yunit=0.01in}
\psset{linewidth=0.005in}
\begin{pspicture}(0,0)(93,93)
\psaxes[dy=10, Dy=1, dx=10, Dx=1, tickstyle=bottom, showorigin=false, labels=none](0,0)(90,90)
\psframe[linecolor=darkgray,fillstyle=solid,fillcolor=lightgray,linewidth=0.02in](7,7)(53,53)
\psframe[linecolor=darkgray,fillstyle=solid,fillcolor=lightgray,linewidth=0.02in](57,57)(93,93)
\pscircle*(10,40){0.04in}
\pscircle*(20,20){0.04in}
\pscircle*(30,50){0.04in}
\pscircle*(40,10){0.04in}
\pscircle*(50,30){0.04in}
\pscircle*(60,90){0.04in}
\pscircle*(70,80){0.04in}
\pscircle*(80,70){0.04in}
\pscircle*(90,60){0.04in}
\end{pspicture}
&\rule{0.5in}{0pt}&
\psset{xunit=0.01in, yunit=0.01in}
\psset{linewidth=0.005in}
\begin{pspicture}(0,0)(93,93)
\psaxes[dy=10, Dy=1, dx=10, Dx=1, tickstyle=bottom, showorigin=false, labels=none](0,0)(90,90)
\psframe[linecolor=darkgray,fillstyle=solid,fillcolor=lightgray,linewidth=0.02in](7,93)(33,67)
\psframe[linecolor=darkgray,fillstyle=solid,fillcolor=lightgray,linewidth=0.02in](37,63)(93,7)
\pscircle*(10,70){0.04in}
\pscircle*(20,90){0.04in}
\pscircle*(30,80){0.04in}
\pscircle*(40,60){0.04in}
\pscircle*(50,30){0.04in}
\pscircle*(60,40){0.04in}
\pscircle*(70,10){0.04in}
\pscircle*(80,50){0.04in}
\pscircle*(90,20){0.04in}
\end{pspicture}
\end{tabular}
\end{center}
\caption{An inflation of $2413$, a sum, and a skew sum.}\label{fig-479832156}
\end{figure}

Every permutation $\pi$ is the inflation of a unique simple permutation, called its \emph{simple quotient}.  If the quotient has length greater than $2$ then the intervals inflating the quotient are uniquely determined by $\pi$ (as Albert and Atkinson~\cite{albert:simple-permutat:}).  If the quotient has length $2$ (i.e., $\pi$ is sum or skew decomposable), then we can enforce uniqueness by insisting that the first interval be sum or skew indecomposable.

A permutation class is \emph{sum closed} if it contains the sum of any two of its members.  Clearly the \textsf{PS}-sortable permutations form a sum closed class because if \textsf{PS} can sort $\pi$ and $\sigma$ then it can sort $\pi\oplus\sigma$ (simply sort the interval corresponding to $\pi$, then sort the interval corresponding to $\sigma$).

\begin{proposition}[Albert, Atkinson, and Vatter~\cite{albert:inflations-of-g:2x4}]
The simple permutations in $\Av(2431,3142)$ are $1$ (trivially), $12$, $21$, and the \emph{parallel alternations} of the form
%\begin{equation}
%\label{eq-par-alt}
$$
246\cdots(2m)135\cdots(2m-1)
$$
%\tag{$\dagger$}
%\end{equation}
for $m\ge 2$.
\end{proposition}

We aim to produce a structural description of $\Av(2431,3142,3241)$ which will allow us to show that every permutation in the class is \textsf{PS}-sortable.  Because each of these basis elements is sum indecomposable, this class is sum closed (as is the class of \textsf{PS}-sortable permutations).  For skew sums $\pi\ominus\sigma$, notice that to avoid $2431$ and $3241$, $\pi$ must avoid $132$ and $213$, but there are no further restrictions (except that $\sigma$ must lie in the larger class, obviously).  The class $\Av(132,213)$ is easily seen to be the \emph{reverse layered permutations}, i.e., those of the form
$$
\iota_1\ominus\cdots\ominus\iota_m
$$
where each $\iota_1,\dots,\iota_m$ is increasing.

It remains to consider inflations of parallel alternations of the form $246\cdots(2m)135\cdots(2m-1)$.    In order to avoid $3241$, all even entries in such a parallel alternation except the greatest may only be inflated by increasing intervals.  Furthermore, in order to avoid $2431$, the greatest even entry also may only be inflated by an increasing interval.  Inflations of the odd entries, however, have no further restrictions.  We have therefore obtained the following structural decomposition of this class.

\begin{proposition}\label{prop-2431-3142-3241-structure}
The class $\Av(2431, 3142, 3241)$ consists precisely of permutations of the form
\begin{enumerate}
\item[(a)] $\pi\oplus\sigma$ where $\pi,\sigma\in\Av(2431, 3142, 3241)$,
\item[(b)] $\pi\ominus\sigma$ where $\pi\in\Av(132,213)$ and $\sigma\in\Av(2431, 3142, 3241)$, and
\item[(c)] inflations of parallel alternations $246\cdots(2m)135\cdots(2m-1)$ for $m\ge 2$ where the even entries are inflated by increasing intervals and the odd entries are inflated by intervals in $\Av(2431, 3142, 3241)$.
\end{enumerate}
\end{proposition}

We now describe how to sort all of the permutations in $\Av(2431,3142,3241)$, verifying that this is indeed the class of \textsf{PS}-sortable permutations.

\begin{theorem}\label{thm-PS-basis}
The \textsf{PS}-sortable permutations are precisely $\Av(2431, 3142, 3241)$.
\end{theorem}
\begin{proof}
Choose an arbitrary $\pi\in\Av(2431, 3142, 3241)$.  Using induction on the length of $\pi$, we show that $\pi$ is \textsf{PS}-sortable.  The base case is trivial, as \textsf{PS} can sort $1$.  If $\pi=\sigma\oplus\tau$ for shorter permutations $\sigma$ and $\tau$ in this class, then by induction, the \textsf{PS} machine can sort (and output) $\sigma$ and then sort $\tau$.  If $\pi$ is skew indecomposable, then we know from Proposition~\ref{prop-2431-3142-3241-structure} that
$$
\pi=\iota_1\ominus\cdots\iota_m\ominus\sigma
$$
for increasing permutations $\iota_1,\dots,\iota_m$ and an arbitrary $\sigma\in\Av(2431, 3142, 3241)$.  To sort permutations of this form, we push each $\iota_k$, in turn, onto the pop stack and then pop it into the stack.  After having performed this operation on $\iota_1,\dots,\iota_m$, we are left with their entries sitting in increasing order in the stack, and thus can (by induction) sort and output $\sigma$ and then output the entries of $\iota_1,\dots,\iota_m$.

Finally suppose $\pi$ is of the form (c) in Proposition~\ref{prop-2431-3142-3241-structure}, so
$$
\pi=246\cdots(2m)135\cdots (2m-1)[\iota_1,\dots,\iota_m,\sigma_1,\dots,\sigma_m]
$$
where $\iota_1,\dots,\iota_m$ are increasing and $\sigma_1,\dots,\sigma_m\in\Av(2431, 3142, 3241)$.  To sort these permutations, we first push all entries of $\iota_1,\dots,\iota_m$ onto the pop stack and then pop them all into the stack.  This leaves the entries of $\iota_1,\dots,\iota_m$ in increasing order on the stack.  We then sort and output $\sigma_1$ (which can be done by induction), then output $\iota_1$, then sort and output $\sigma_2$, then output $\iota_2$, and so on.
\end{proof}

Proposition~\ref{prop-2431-3142-3241-structure} also leads almost immediately to the enumeration of this class.

\begin{theorem}\label{thm-PS-enum}
The \textsf{PS}-sortable permutations are enumerated by the generating function
$$
\frac{1-3x+2x^2-\sqrt{1-6x+5x^2}}{2x(2-x)},
%(1-3*x+2*x^2-sqrt(1-6*x+5*x^2))/(2*x*(2-x));
$$
sequence \OEISlink{A033321} in the \OEISref.
\end{theorem}
\begin{proof}
Let $f$ denote the generating function for the \textsf{PS}-sortable permutations and $f_\oplus$ (resp., $f_\ominus$) denote the generating function for the sum (resp., skew) decomposable \textsf{PS}-sortable permutations.

The sum indecomposable \textsf{PS}-sortable permutations are therefore counted by $f-f_\oplus$.  Because every sum decomposable permutation can be expressed uniquely as the sum of a sum indecomposable permutation with an arbitrary permutation, we see that $f_\oplus=(f-f_\oplus)f$.  Solving this for $f_\oplus$ yields
$$
f_\oplus=\frac{f^2}{1+f}.
$$

Now consider skew decomposable permutations.  The reverse layered permutations contain a unique skew indecomposable permutation of each length (the increasing permutation), so by Proposition~\ref{prop-2431-3142-3241-structure}, the contribution of skew decomposable permutations is
$$
f_\ominus=\frac{xf}{1-x}.
$$

Finally, the contribution of permutations of the form (c) in Proposition~\ref{prop-2431-3142-3241-structure} is given by
$$
\sum_{m\ge 2} \left(\frac{xf}{1-x}\right)^m=\frac{(xf)^2}{(1-x)(1-x-xf)}.
$$

Combining these quantities (and the contribution of the permutation $1$) we have
$$
f
%=
%f_\oplus+\mbox{g.f. for inflations of parallel alternations}
=
x+
\frac{f^2}{1+f}
+
\frac{xf}{1-x}
+
\frac{(xf)^2}{(1-x)(1-x-xf)},
$$
and solving this for $f$ completes the proof of the theorem.
\end{proof}

The generating function from Theorem~\ref{thm-PS-enum} has arisen at least twice before in the study of permutation patterns.  The Theory of Computing Research Group at the University of Otago~\cite{albert:sorting-classes:} showed that it counts $\Av(2431, 4231, 4321)$, while Brignall, Huczynska, and Vatter~\cite{brignall:simple-permutat:} showed that it enumerates $\Av(2143, 2413, 3142)$.  None of these three classes are symmetries of each other so this is example of ``Wilf-equivalence''.

% TODO: Add note to OEIS.

%
%
%
%

\section{With a Queue Between --- \textsf{PQS}}\label{sec-pop-with}

There is another way to characterize the permutations which are \textsf{PS}-sortable, which we introduce now because we will use it to characterize the \textsf{PQS}-sortable permutations.  A {\it divided permutation\/} is a permutation equipped with one or more dividers $|$, i.e., $\pi_1|\pi_2|\cdots|\pi_t$.  We refer to $\pi_1|\pi_2|\cdots|\pi_t$ as a {\it division\/} of the concatenated permutation $\pi_1\pi_2\cdots\pi_t$, and we refer to each $\pi_i$ as a {\it block\/} of this division.  We say that the divided permutation $\sigma_1|\sigma_2|\cdots|\sigma_s$ is contained in the divided permutation $\pi_1|\pi_2|\cdots|\pi_t$ if $\pi_1\pi_2\cdots\pi_t$ contains a subsequence order isomorphic to $\sigma_1\sigma_2\cdots\sigma_s$ in which the entries corresponding to each $\sigma_i$ come from the same block, and no other entries of this subsequence come from that block.  For example:
\begin{itemize}
\item $513|4|2$ contains $32|1$ because of the subsequence $532$, but
\item $32|1$ is not contained in $51|34|2$ despite the subsequence $532$.
\end{itemize}
In particular, if $\sigma$ contains no dividers, then $\sigma$ is contained in 
$\pi_1|\pi_2|\cdots|\pi_t$ if and only if $\sigma$ is contained in a single block $\pi_i$.

\begin{proposition}\label{prop-without-queue-divided}
The permutation $\pi$ can be sorted by \textsf{PS} if and only if divisions can be added to $\pi$ to obtain a divided permutation which avoids $21$, $2|13$, and $2|3|1$.
\end{proposition}
\begin{proof}
We view the divisions as marking the moments when we transfer all contents of the pop stack to the stack (with a final transfer occurring at the end of reading the permutation).  Thus if there is no such division of $\pi$, in any sorting of this permutation there will come a time when either:
\begin{itemize}
\item the pop stack contains an increase when read from top to bottom ($21$),
\item the stack contains an entry which lies between two entries of the pop stack in value ($2|13$), or
\item the stack contains a decrease when read from top to bottom ($2|3|1$).
\end{itemize}
Any of these three situations will cause the sorting to fail.

Conversely, suppose that divisions can be added to $\pi$ to obtain a divided permutation which avoids $21$, $2|13$, and $2|3|1$.  To show $\pi$ is sortable, we need to show that none of the transfers dictated by these divisions forces an inversion in the stack.  We know a set of entries moved by a single transfer will not form an inversion in the stack since the divided permutation avoids $21$.  The only other way an inversion could be forced within the stack is if there is an entry $2$ that was previously transferred to the stack, an entry $3$ that is transferred to the stack later, and an entry $1$ that is transferred at the same time as or later than the $3$ (and thus not allowing the $2$ to be output before the $3$ enters the stack).  Notice these \textsf{PS} movements imply that there is a $2|13$, a $2|31$ (which means there is a $21$), or a $2|3|1$.  This shows that $\pi$ is sortable by \textsf{PS}.
\end{proof}

The analogue of Proposition~\ref{prop-without-queue-divided} for the \textsf{PQS} machine is the following.

\begin{proposition}\label{prop-with-queue-divided}
The permutation $\pi$ can be sorted by \textsf{PQS} if and only if divisions can be added to $\pi$ to obtain a divided permutation which avoids $132$, $2|13$, $32|1$, and $2|3|1$.
\end{proposition}
\begin{proof}
A permutation can be sorted with a stack if and only if it avoids $231$, so we need to show that we can fill the queue between the pop stack and the stack with a $231$-avoiding permutation if and only if $\pi$ can be divided in the manner specified.

First suppose that $\pi$ can be divided as $\pi_1|\pi_2|\cdots|\pi_t$ so that this division avoids the four divided permutations $132$, $2|13$, $32|1$, and $2|3|1$.  Consider pushing each element of $\pi_1$ into the pop stack and then popping the entire pop stack into the queue, then pushing each element of $\pi_2$ onto the pop stack and then popping the entire pop stack into the queue, and so on.  This fills the queue with the permutation $\pi_1^{\rm r}\pi_2^{\rm r}\cdots\pi_t^{\rm r}$.  Now consider the four different ways this permutation could contain $231$: if all three entries are in $\pi_i^{\rm r}$ then $\pi_i$ contains $132$, if the first entry is in $\pi_i^{\rm r}$ and the other two are in $\pi_j^{\rm r}$ for $i<j$ then $\pi_i|\pi_j$ contains $2|13$, if the first two entries are in $\pi_i^{\rm r}$ and the last is in $\pi_j^{\rm r}$ for $i<j$ then $\pi_i|\pi_j$ contains $32|1$, and finally, if all three entries are in different blocks, then $\pi_1|\pi_2|\cdots|\pi_t$ contains $2|3|1$.  The other direction follows immediately.
\end{proof}

The \textsf{PS}-sortable permutations can be characterized by a finite number of forbidden divided patterns (Proposition~\ref{prop-without-queue-divided}) and by a finite basis (Theorem~\ref{thm-PS-basis}).  However, it does \emph{not} follow that every class defined by finitely many forbidden divided patterns has a finite basis, as we show in the next section.  Nor is it apparent how to convert such a list of divided patterns into a basis.  In the case of the \textsf{PQS} machine, we have not been able to verify that it is finitely based, although computations performed by Michael Albert strongly suggests this to be the case.

\begin{conjecture}\label{conj-PQS-basis}
The class of \textsf{PQS}-sorting permutations consists of $108$ permutations, all of length at most $9$.
\end{conjecture}

The enumeration of this class is
$$
1, 2, 6, 24, 120, 685, 4148, 25661, 159829, 997870, \dots,
$$
sequence \OEISlink{A214611} in the \OEISref.

\section{Avoiding Divided Permutations in General}\label{sec-pop-general}

As remarked in Section~\ref{sec-pop-without}, classes defined by finitely many divided permutations need not be finitely based.  Here we give an example.  Our basis will contain the infinite antichain of permutations referred to as $U$:
\begin{eqnarray*}
u_1&=&2,3,5,1,6,7,4\\
u_2&=&2,3,5,1,7,4,8,9,6\\
u_3&=&2,3,5,1,7,4,9,6,10,11,8\\
&\vdots&\\
u_k&=&2,3,5,1,7,4,9,6,11,8,\dots,2k+3,2k,2k+4,2k+5,2k+2.
\end{eqnarray*}
Each member of $U$ has precisely two copies of $2341$: its first four entries, and the first, second, third, and fifth entries from the right, and this observation can be used to prove that $U$ is indeed an antichain (see Atkinson, Murphy, and Ru\v{s}kuc~\cite{atkinson:partially-well-:} for such a proof).  We call the first four entries in $u_k$ the \emph{head}, the last five entries the \emph{tail}, and the entries between the \emph{midsection}.

\begin{proposition}\label{prop-div-inf-basis}
Every member of the infinite antichain $U$ is a basis element for the permutation class defined by avoiding the divided permutations
\begin{enumerate}
\item[(a)] $2341$.
\item[(b)] $234|1$, $23|4|1$, $2|34|1$, $2|3|4|1$,
\item[(c)] $314|2$, $31|42$, and $31|4|2$.
\end{enumerate}
\end{proposition}

\begin{proof}
Consider any member of the antichain $U$.  In order to avoid $2341$, the head --- $2351$ --- must be divided.  However, to avoid $234|1$, $23|4|1$, $2|34|1$, and $2|3|4|1$, this division cannot occur between the $5$ and the $1$.  Therefore there is a block containing $51$.  Now consider the four entries starting with $51$, which in $u_k$ for $k\ge 2$ consist of $5174$.  As these entries are order isomorphic to $3142$ and the $51$ block is not divided, in order to avoid $314|2$, $31|42$, and $31|4|2$, they cannot be divided.  This propagates throughout the midsection of $u_k$, and at the end of the process, we see that then entries $5,1,7,4,9,6,11,8,\dots,2k+1,2k-2$ all lie in the same block.  Now consider the tail.  Because the entries $2k+1, 2k-2, 2k+3, 2k$ are order isomorphic to $3142$ and the $2k+1$ and $2k-2$ are not divided, there cannot be a division between then.  The same argument shows that the entries $2k+3, 2k, 2k+5, 2k+2$ lie in the same block.  However, this implies that this block contains the entire tail, which is order isomorphic to $2341$.

We must now argue that by removing any entry of such a permutation we obtain a permutation in the class, which follows  from a case analysis.  Typically, when an entry is removed one would relabel the remaining entries so that they consist of the numbers $1$ through $n-1$, but it is easier to make this argument without relabeling, so, for example, if the $1$ is removed, we talk about the permutation beginning with $235$.

Suppose an entry of the head is removed.  In this case, the remaining elements of the head do not form a copy of $2341$, so we are able to insert a division immediately preceding the $1$ (assuming the $1$ was not removed), or immediately following the $5$ (assuming the $5$ was not removed).  Because of this division, we can then insert divisions between every two entries in the midsection.  This then allows us to add a division in the tail between $2k+3$ and $2k$, and the resulting division avoids the divided permutations desired.  For example, if the $3$ is removed from $u_5$ we have the division
\[
2\ \cancel{3}\ 5\ |\ 1\ |\ 7\ |\ 4\ |\ 9\ |\ 6\ |\ 11\ |\ 8\ |\ 13\ |\ 10\ 14\ 15\ 12.
\]

Suppose now that an entry from the midsection is removed.  In this case the permutation splits into a direct sum of two shorter permutations, one consisting of the head and forward midsection, and the other consisting of the rear midsection and the tail.  We can then add a division in the first sum component the $2$ and the $3$.  In the second component, we add divisions between all entries in the rear midsection, and then between the $2k+3$ and $2k$ in the tail.  It is straight-forward to check that the resulting division avoids the desired divided permutations.  For example, if the $9$ is removed from $u_5$ we have the division
\[
2\ |\ 3\ 5\ 1\ \underline{7\ 4}\ \cancel{9}\ \overline{6\ |\ 11\ |\ 8}\ |\ 13\ |\ 10\ 14\ 15\ 12,
\]
where forward and rear midsections are denoted by under- and over-lining, respectively.

To complete the proof, suppose that an entry from the tail is removed.  If this entry is part of the copy of $2341$, then we can simply add a division between the $2$ and $3$ at the beginning of the permutation.  For example, if the $14$ is removed from $u_5$, we obtain
\[
2\ |\ 3\ 5\ 1\ 7\ 4\ 9\ 6\ 11\ 8\ 13\ 10\ \cancel{14}\ 15\ 12.
\]
Otherwise, the entry $2k$ was removed.  In this case, the $2k+3$ and $2k+4$ are not involved in a copy of $3142$, and so we can add a division between them, and also between the $2$ and the $3$.  For example, if the $10$ is removed from $u_5$, we obtain
\[
2\ |\ 3\ 5\ 1\ 7\ 4\ 9\ 6\ 11\ 8\ 13\ |\ \cancel{10}\ 14\ 15\ 12.
\]
This final case completes the proof.
\end{proof}

\section{Concluding Remarks}

Note that the pop stack can never contain a noninversion (when read from top to bottom) in the \textsf{PS} machine.  Thus \textsf{PS} sorting is a special case of sorting with a decreasing stack followed by an increasing stack, the \textsf{DI} machine, which has been studied by Smith~\cite{smith:a-decreasing-st:}.  However, there is no relation between \textsf{DI}-sortable and \textsf{PQS}-sortable permutations --- $3142$ can be sorted by \textsf{PQS} but not \textsf{DI}, while $465132$ can be sorted by \textsf{DI} but not \textsf{PQS}.

We have demonstrated with Proposition~\ref{prop-div-inf-basis} that even though the \textsf{PQS} sortable permutations can be characterized by finitely many divided patterns, this alone does not imply that this class has a finite basis.  However, there is another possible generalization of this problem.  Let us say that the permutation $\sigma$ can be obtained from $\pi$ by \emph{local reversals} if $\pi=\pi_1\pi_2\cdots\pi_t$ and $\sigma=\pi_1^{\rm r}\pi_2^{\rm r}\cdots\pi_t^{\rm r}$.  Thus the \textsf{PQS}-sortable permutations are those that can be obtained by local reversals from the $231$-avoiding permutations.

\begin{question}\label{question-local-reversals}
Let $\C$ be a permutation class and $\D$ the class of all permutations that can be obtained by local reversals from members of $\C$.  If $\C$ is finitely based, must $\D$ also be finitely based?
\end{question}

Our suspicion is that the answer to Question~\ref{question-local-reversals} is ``no''.  This suspicion is based on an example found by the Theory of Computing Research Group at the University of Otago~\cite{albert:compositions-of:}.  They considered a machine denoted $\textsf{T}$,called a \emph{transposition switch}.  Given a permutation $\pi$, $\textsf{T}$ returns the set of permutations which can be generated from $\pi$ by disjoint adjacent transpositions of entries of $\pi$.  This operation is then extended to sets (and thus, classes) of permutations in the natural way.  The Otago group showed that while $\textsf{T}^4(\Av(21))$ is finitely based, $\textsf{T}^k(\Av(21))$ is infinitely based for all $k\ge 5$.  Thus $\textsf{T}$ does not preserve finite bases.

\bigskip
\noindent{\bf Acknowledgments:} The analysis of the \textsf{PS} machine in Section~\ref{sec-pop-without} was aided by Michael Albert's \emph{PermLab} program~\cite{PermLab1.0}.  The authors would also like to thank Albert for computing the basis in Conjecture~\ref{conj-PQS-basis} and drawing our attention to \cite{albert:compositions-of:}, and Daniel Rose for his \LaTeX\ macros for drawing two stacks in series.

\bibliographystyle{acm}
\bibliography{../refs}

\end{document}